\documentclass{amsart}

\title{Free amalgamation and automorphism groups}

\author{Andreas Baudisch}
\address{
Institut f\"ur Mathematik,Humboldt-Universit\"at zu Berlin,
  D-10099 Berlin, Germany}

\email{baudisch@mathematik.hu-berlin.de}
\date{04.06.2014}


\usepackage[latin1]{inputenc}

\usepackage{amssymb, amsmath, tikz}
\usetikzlibrary{matrix,arrows, calc}
\usetikzlibrary{fit}
\usetikzlibrary{positioning}

\usepackage{enumerate,xspace}

\RequirePackage{ifpdf}
\ifpdf
   \usepackage[pdftex]{hyperref}
\else
   \usepackage[hypertex]{hyperref}
\fi

\newcommand{\Frai}{ Fra\"{\i}ss\'e  }
\theoremstyle{plain}
\newtheorem{theorem}{Theorem}[section]
\newtheorem{cor}[theorem]{Corollary}

\newtheorem{lemma}[theorem]{Lemma}

\theoremstyle{definition}

\newtheorem{fact}[theorem]{Fact}
\newtheorem{definition}[theorem]{Definition}

\newcommand{\nc}{\newcommand}
\nc{\Q}{\mathbb{Q}}
\nc{\C}{\mathfrak{C}}
\nc{\G}{\mathfrak{G}}
\nc{\N}{\mathbb{N}}
\nc{\cb}{\operatorname{Cb}}
\nc{\Kom}{\mathbb{K}} 
\nc{\K}{\mathcal{K}}
\nc{\J}{\mathcal{J}}
\nc{\tp}{\operatorname{tp}}
\nc{\stp}{\operatorname{stp}}
\nc{\RM}{\operatorname{RM}}

\def\F{\mathbb F}

\nc{\Mr}{\operatorname{MR}}
\nc{\U}{\operatorname{U}}
\nc{\p}{\operatorname{p}}
\nc{\A}{\mathcal{A}} 
\nc{\Cent}{\operatorname{C}}

\nc{\lmto}[1]{\xrightarrow[{#1}]{}} 
\nc{\ggen}[1]{\langle #1\rangle} 
\nc{\rest}[2]{#1\!\!\upharpoonright_{#2}}

\renewcommand{\iff}{if and only if\xspace}

\def\Ind#1#2{#1\setbox0=\hbox{$#1x$}\kern\wd0\hbox to 0pt{\hss$#1\mid$\hss}
\lower.9\ht0\hbox to 0pt{\hss$#1\smile$\hss}\kern\wd0}
\def\Notind#1#2{#1\setbox0=\hbox{$#1x$}\kern\wd0\hbox to
0pt{\mathchardef\nn="0236\hss$#1\nn$\kern1.4\wd0\hss}\hbox
to 0pt{\hss$#1\mid$\hss}\lower.9\ht0
\hbox to 0pt{\hss$#1\smile$\hss}\kern\wd0}
\def\ind{\mathop{\mathpalette\Ind{}}}

\begin{document}
\begin{abstract}
Let $L$ be a countable  elementary language, $M_0$ be a \Frai limit.  We consider  free  amalgamation for L-structures where L is arbitrary.  If  free amalgamation for  finitely generated substructures  exits in $M_0$, then it is a stationary independece relation in the sense of K.Tent and M.Ziegler \cite{TZ12b}.  Therefore $Aut(M_0)$ is universal for $Aut(M)$ for all  substructures $M$ of $M_0$. This follows by a result of  I.M\"uller \cite{Mue13}  We show that c-nilpotent graded Lie algebras over a finite field and c-nilpotent groups of exponent p ($c < p$) with extra predicates for a central Lazard series provide examples. We replace the proof in \cite{Bau04} of the  amalgamation of c-nilpotent graded Lie algebras over a field by a correct one. 
 
\end{abstract}
\maketitle

\section{Introduction}
Let $L$ be a countable elementary language. Let $M_0$ be a \Frai limit in $L$.   Eric Jaligot (\cite{Jal07} asked whether the group  $Aut(M_0)$ of automorphisms of $M_0$ is universal for all groups $Aut(M)$, where $M$ is a substructure of $M_0$. He proved this for random tournaments. The first example is the Urysohn space \cite{Usp90}. Also for \Frai limits in relational languages it is true \cite{Bil12}, if there is free amalgamation for the age.

We introduce the free amalgam $A \otimes_B C$ for a class $\J$ of L-structures, where $L$ is arbitrary (Chapter 2). 
In this  general stuation we use other considerations than in \cite{Bil12}

If  we have free amalgamation in the age of a \Frai limit $M_0$, we can define $A \ind_B C$ for finite subsets of $M_0$ by 
\[\langle ABC \rangle = \langle A B\rangle \otimes_{\langle B \rangle} \langle BC \rangle .\]
$\langle X \rangle$ denotes the substructure generated by $X$.
We show that this is a stationary independence relation in $M_0$ in the sense of  K.Tent and M.Ziegler \cite{TZ12b}.  We realize, that {\bf Mon} is a consequence of the remaining properties in general.
If furthermore the age of $M_0$ is uniformly locally finite, then we have free amalgamation for the  substructures of the monster model $\C$ of $Th(M_0)$ and it gives a stationary independence relation for the substructures of $\C$. That means we have all properties of non-forking in a stable theory except local character and boudedness is replaced by the stronger  property stationarity. But the examples we discuss below have the tree property of the second kind.

We use a new idea, developed by Isabel M\"uller in  \cite{Mue13}. Let $M_0$ be a \Frai limit as above.  She proved, that the existence of a stationary independence relation  for finite subsets of $M_0$ in the sense of K.Tent und M.Ziegler \cite{TZ12b} implies the universality of $Aut(M_0)$ for all $Aut(M)$ where $M$ is a substructure of $M_0$. The stationary independence relation is used to reconstruct the \Frai limit $M_0$ from a given substructure $M$ using so-called Katetov extensions. In general the embedding of $M$ in $M_0$ will change. 
We apply S.M\"uller's result to \Frai limits $M_0$ with free amalgamation and obtain the universality of $Aut(M_0)$ for all groups $Aut(M)$, where $M \subseteq M_0$  (chapter 3).

In chapter 4 we reprove the existence of the free amalgam for the class of c-nilpotent graded Lie algebras over a  field $K$ in a language with extra predicates for the graduation. Unfortunately  the proof of this result in \cite{Bau04} is incorrect. The existence of the free amalgam for all graded Lie algebras over a given field  follows. We get a \Frai limit $M_0$ of the finitely generated c-nilpotent graded Lie algebras over a finite field $K$. Then the free amalgam  gives  a stationary independence relation in $M_0$. $Aut(M_0)$  is universal for all $ \{ Aut(M) : M \subseteq M_0 \}$.  For c-nilpotent graded assoziative algebras even amalgams do not exists in general, as a counterexample in chapter 5 shows.

In the last chapter we consider c-nilpotent groups of exponent $p > c$ with extra predicates for a central Laszard series. As shown in \cite{Bau04}  the results for graded Lie algebras imply the existence of the free amalgam  for all these groups. The  \Frai limit $G^U_0$ exists for these groups and the free amalgam gives a stationary independence relation. Hence $Aut(G^U_0)$ is universal for $\{Aut(G^U) : G^U \subseteq G^U_0 \}$. Let $G_0$ be the reduct of $G^U_0$ to the language of group theory. Using the lower central series we can transform each c-nilpotent group of exponent $p > c$ to a structure of the extended language. Hence $G_0$ is universal for all these groups. Since the upper and lower central series in $G_0$ coincide, the extra predicates are 0-definable in $G_0$. Therefore $Aut(G_0)$ is universal for all $Aut(G)$ where $G$ is a at most countable c-nilpotent group of exponent $p > c$.   Note that the elementary theories of $M_0$ (Lie algebras), $G^U_0$, and $G_0$ have the tree property of the second kind (see \cite{Bau13}).

I would like to thank Martin Ziegler for helpful discussions of the results, especially for a shorter proof of Lemma \ref{mon}.

\section{Free  Amalgamation}
Let $\K$ be a class of finitely generated $L$-structures. $\K$ is the age (or skeleton) of a $L$-structure $M$, if $\K$ is the class of all $L$-structures that are isomorphic to a finitely generated substructure of $M$. In this paper $L$ and $\K$ are always countable.

\begin{definition}
$M$ is $\K$-saturated, if $\K$ is the age of $M$ and if for all $B$, $A$ in $\K$ and all embeddings 
$f_0: B \to M$, $f_1: B \to A$ there is an embedding $g: A \to M$ such that $f_0 = g \circ f_1$.
\end{definition}

Then the following is well-known:

\begin{fact}\label{fact2.1}
Countable $\K$-saturated structures  are isomorphic. Let $M_0$ be a $\K$-saturated structure. It is ultrahonogeneous. That means an isomorphism between finitely generated substructures of $M_0$ can be extened to an automorphism. Conversely countable ultrahomogeneous structures $M_0$ are $\K$ - saturated, where $\K$ is the age of  $M_0$.
$M_0$ is $\K$-universal: Every countable L-structure with an age included in $\K$ can be embedded.
\end{fact}

This fact implies that the quantifier free n-type of an n-tuple implies the full n-type in $M_0$. But this is not quantifier elimination for $Th(M_0)$.

\begin{fact}\label{fact2.2}
There is a countable $\K$-saturated $L$-structure $M_0$ if and only if $\K$ has the following properties:

\begin{description}

\item[HP] {\em Heredity Property} For $A$ in $\K$ we have $age(A) \subseteq \K$. 

\item[JEP] {\em Joint Embedding Property} For $A$ and $C$ in $\K$ there are some $D \in \K$ and embeddings $f_0 : A \to D$ and $f_1: C \to D$.

\item[AP] {\em Amalgamation Property} Assume  $g_0: B \to A$ and $g_1: B \to C$ are embeddings for
$A, B, C \in \K$. Then there are some $D$ in $\K$ and embeddings $f_0: A \to D$ and $f_1: C \to D$ such that $f_0 \circ g_0 = f_1 \circ g_1$ for $B$.

\end{description}
\end{fact}

$M_0$ in \ref{fact2.2} is called the \Frai limit of $\K$. By \ref{fact2.1} it is unique up to isomorphisms.  

\begin{definition}  {\bf APS} We have the strong amalgamation property  for $\K$ if in {\bf AP} 
$f_0(A) \cap f_1(C) = f_0 \circ g_0 (B) = f_1 \circ g_1(B)$ holds. 
\end{definition}

\begin{fact}\label{fact2.3}
Assume $L$ is finite, $\K$ is uniformly locally finite, and a $\K$-saturated $L$-structure $M_0$ exists. Then $Th(M_0)$ is $\aleph_0$-categorical and allows the elimination of quantifiers.
\end{fact}

For the next considerations we assume again, that $L$ is countable and $\J$ is a  class of  $L$-structures .

\begin{definition}\label{freeamal} Let  $A, B, C, D \in \J$ and assume that $B$ is a common substructure of
$A$ and $C$. If $D$ is generated by $A$ and $C$ with $A \cap C = B$, then 
$D$ is the  free amalgam of $A$ and $C$ over $B$ (short $D = A \otimes_B C$) in $\J$, if for  all  homomorphisms
$f: A \to E$ and $g:C \to E$  into some $E \in \J$ with 
$f(b) = g(b)$ for $b \in B$ there is a homomorphism $h: D \to E$ that extends $f$ and $g$.

$\J$ is closed under free amalgamtion, if for  $A,B,C \in \J$ and embeddings $g_0 : B \to A$ and $g_1 : B \to C$, there exists a  free amalgam $A' \otimes_{B'} C'$  in $\J$ and isomorphisms $f_0 : A \to A'$ and $f_1 : C \to C'$ , such that $f_0 \circ g_0(b) = f_1 \circ g_1(b)$ for $b \in B$ maps $B$ onto $B'$ . 
\end{definition}

The free amalgam is a strong amalgam by definition.
Note that $A \otimes_B C$ is uniquely determined up to isomorphisms, if it exists. 
If $L$ is relational and $\J$ is the class of all $L$-structures, then the  free amalgam exists. Its domain is the union of $A$ and $C$ with intersection $B$ and the only relations are the old relations from $A$ and $C$.
In this paper we will consider  free amalgams in the class graded Lie algebras over  fields and in the class of c-nilpotent groups of exponent p ($c < p$) with extra predicates for a central Lazard series.
\medskip

We add new constant symbols $e_a$ for $a \in A\setminus B$ $e_b$ for $b \in B$ and $e_c$ for  $c \in C\setminus B$ to the language $L$ and assume that we have the same symbols for the elements of $B$ as a substructure of $A$ and of $C$ respectively. Using these constant symbols we define the diagrams $Dia(A)$ and $Dia(C)$ - the sets of all atomic sentences and negated atomic sentences in this enriched language that are true in $A$ respectively in $C$, if we  interprete the new constant symbols by the elements they represent. 

\begin{definition}\label{D: sigma}
Let $\Sigma_{\J} (A,B,C)$ be the union of $Dia(A)$ and $Dia(C)$ with all negated atomic sentences
$e_a \not= e_c$ for $a \in (A \setminus B)$ and $c \in (C \setminus B)$ and all negated atomis sentences
 $\neg\phi(\bar{e}_{\bar{a}}, \bar{e}_{\bar{b}}, \bar{e}_{\bar{c}})$, where
$\bar{a} \subseteq A $
$\bar{b} \subseteq B$, and $\bar{c} \subseteq C $ and there are
homomorphisms $f$ and $g$ of $A$ and $C$ repectively into some $E \in \J$ with $f(b) = g(b)$ for  $b \in B$, such that $E \models \neg\phi(\bar{f}(\bar{a}, \bar{f}(\bar{b}), \bar{g}(\bar{c}))$.

\end{definition}

\begin{lemma}\label{sigma}
\begin{enumerate}

\item Assume $\J$ is closed under substructures. For $A,B,C \in \J$ the  free amalgam $A\otimes_B C$ exists in $\J$  if and only if $\Sigma_{\J}(A,B,C)$ has a model in $\J$.

\item Let $\J$ be an $\forall$-elementary class such that substructures of finitely generated structures in $\J$ are again finitely generated. Then $\J$ is closed under  free amalgamation \iff the finitely generated structures in $\J$ are closed under free amalgamation. 

\item Let $L$ be finite and $\K$ be countable class of finitely generated L-structures that are uniformly locally finite. Assume a $\K$-saturated model $M_0$ exists. Let $\J$ be the class of the substructures of the models of $Th(M_0)$.  
If $\K$ is closed under free amalgamation, then $\J$ is closed under free amalgamation.

\end{enumerate}
\end{lemma}

\begin{proof}
\begin{enumerate}
\item $A \otimes_B C$ models $\Sigma_{\J}(A,B,C)$. If $M$ is a model of
$\Sigma_{\J}(A,B,C)$ in $\J$, then let $D$ be the substructure of $M$ generated by the interpretations of the constant symbols $e_a, e_b, e_c$. $D$ is in $\J$ by assumption. $D$ is a strong amalgam of $A$ and $C$ over $B$. $D$ is free, since 
$\Sigma_{\J}(A,B,C)$ contains  the set of conditions we need to extend every given pair of homomorphisms.

\item To show the non-trivial direction it is sufficient to prove that $\Sigma_{\J}(A,B,C)$ is consistent with $Th(\J)$, since $\J$ is closed under substructures. Because $\J$ is elementary we can use compactness. Let $\Sigma_0$ be a finte subset of   $\Sigma_{\J}(A,B,C)$. Let $A^0$ be the substructure of $A$ generated by all elements of $A$, that occure in a formula of $\Sigma_0$ and $C_0$ the substructure of $C$ generated by all elements of $C$, that occure in a formula  of $\Sigma_0$. Let $B^1$ be $\langle (B \cap A^0), (B \cap C^0)\rangle$. $B^1 \subseteq B$.  By assumption $B^1$ is finitely generated. Then $A^1 = \langle A^0 , B^1 \rangle$ and 
$C^1 = \langle C^0 , B^1 \rangle$ are finitely generated. $B^1$ is a common substructure of $A^1$ and  
$C^1$. Since $\J$ is $\forall$- elementary $A^1$, $B^1$, and $C^1$ are in $\J$. By assumption 
$A^1 \otimes_{B^1} C^1 = D^1$ exists in $\J$. We claim that $D^1$ is a model of $\Sigma_0$.
The formulas from $Dia(A)$ and $Dia(C)$ in $\Sigma_0$ are fulfilled in $D^1$. Assume we have $\neg\phi(\bar{e}_{\bar{a}}, \bar{e}_{\bar{b}}, \bar{e}_{\bar{c}})$ in $\Sigma_0$, where
$\bar{a} \subseteq A \setminus B$,
$\bar{b} \subseteq B$, and $\bar{c} \subseteq C \setminus B$ and furthermore
homomorphisms $f$ and $g$ of $A$ and $C$ repectively into some $E \in \J$ with $f(b) = g(b)$ for  $b \in B$, such that $E \models \neg\phi(\bar{f}(\bar{a}, \bar{f}(\bar{b}), \bar{g}(\bar{c}))$.  If we consider the restriction of $f$ to $A^1$ and of $g$ to $C^1$, then $f(b) = g(b)$ for $b \in B^1$. By the definition of the free amalgam 
\[D^1 \models \neg\phi(\bar{e}_{\bar{a}}, \bar{e}_{\bar{b}}, \bar{e}_{\bar{c}}),\]
as desired. Formulas $e_a \not= e_c$ from $\Sigma_0$ are fulfilled in $D^1$.

\item $\J$ is $\forall$ - elementary. Since $\K$ is uniformly locally finite and $L$ is finite, $\K$ is the class of finite structures in $\J$. We apply (2). 

\end{enumerate}
\end{proof}

\section{Stationary independence and universal automorphism groups} 
Let $L$ be countable.
K.Tent and M.Ziegler  defined a stationary independence relation for the investigation auf automorphism groups in \cite{TZ12b}.  We consider  finite subsets $A, B, C ,D$ of a $L$-structure $M$.

\begin{definition}\label{statind}  
A relation $A \ind_B C$ for finite subsets of $M$ is called  a stationary independence relation  in $M$ if it fulfils the following properties. 
\begin{description}
\item [Inv] {\em Invariance} $A\ind_B C$ depends only on the type of $A,B,C$.

\item[Mon]{\em Monotonicity} $A \ind_B CD$ implies $A \ind_B C$ and $A\ind_{BC} D$.

\item[Trans] {\em Transitivity} $A \ind_B C$ and $A \ind_{BC} D$ imply $A \ind_B CD$.

\item[Sym] {\em Symmetry} $A \ind_B C$ if and only if $C\ind_B A$.

\item[Ex] {\em Existence} For $A, B, C$ there is some $A'$ in $M$ such that $\tp(A/B) = \tp(A'/B)$ and 
$A' \ind_B C$.

\item[Stat] {\em Stationarity}  If  $\tp(A/B) = \tp((A'/B)$, $ A \ind_B C$, and $A' \ind_B C$, then 
$\tp(A/BC) = \tp(A'/BC)$.
\end{description}
\end{definition}

\begin{lemma}\label{mon}
Let  $A \ind_B C$ be a relation on finite subsets of $M$, that satisfies all properties of a stationary independence relation except {\bf Mon}. Then {\bf Mon} follows. 
\end{lemma}
\begin{proof}
We asume $A \ind_B CD$. Applying {\bf Ex} we get $A'$, such that $A' \ind_B C$ and  $tp(A'/B) = tp(A/B)$. Again by {\bf Ex} there is some $A''$ such that $tp(A''/BC) = tp(A'/BC)$ and $A'' \ind_{BC} D$. By {\bf Inv} $A'' \ind_B C$. By $\bf Trans$ $A'' \ind_B CD$. Since $tp(A''/B) = tp(A/B)$
{\bf Stat} implies $tp(A''/BCD) = tp(A/BCD)$. The assertion follows from {\bf Inv}.
\end{proof}

I. M\"uller combined the existence of a stationary independence relation with Kat\u{e}tov's construction \cite{Mue13}.
She proved:

\begin{theorem}\label{mueller}
If $M_0$ is a \Frai limit and there exists a stationary independence relation in $M_0$, then $Aut(M_0)$ is universal for all $Aut(N)$, where $N$ is a substructure of $M_0$. 
\end{theorem}

We will see that free amalgams provide a stationary independence relation.

\begin{theorem}\label{free=ind} Let $M_0$ be a \Frai limit. We assume that
the  free amalgam of finitely generated substructures  of $M_0$ exists in $M_0$ and 
define for finite subsets $A, B, C$ of $M_0$:
\[ A \ind_B C\] if and only if  
\[\langle A B C \rangle = \langle A B \rangle \otimes_{\langle B \rangle} \langle B C \rangle .\]
Then $\ind$ is a stationary independence relation in $\C$.
\end{theorem}

\begin{proof}
By definition $A \ind_B C$ if and only if $\langle A B\rangle \ind_{\langle B \rangle} \langle B C \rangle$.
Hence we assume w.l.o.g. that $A, B, C $ are finitely generated substructures.
\begin{description}
\item[Inv] It is clear since the free amalgam is uniquely determined by its isomorphism type.

\item[Sym] It follows directly from the definition.

\item[Ex] Since the class of finitely generated substuctures of $M_0$ is closed under free amalgamation and $M_0$ is age($M_0$) - saturated we get {\bf Ex}.

\item[Stat] It is a consequence of the uniqueness of the free amalgam and ultrahomogeneity. 

\item[Trans] By assumption 
$\langle ABCD \rangle = (\langle AB \rangle \otimes_B \langle BC \rangle) \otimes_{\langle BC \rangle} \langle BCD \rangle $. We show  that this structure is the free amalgam of $\langle A B \rangle$ and $\langle B C D \rangle$ over $B$. Let $G$ be a structure in age($M_0)$.
Let $f_0$ be a homomorphism of $\langle AB \rangle $ into $G $ and 
$f_1$ be a homomorphism of $ \langle BCD \rangle$ into $G$ such that $f_0 (b) = f_1 (b)$ for $b \in B$.
By $A \ind_B C$ there is a homomorphism $g$ of $\langle ABC \rangle$ into $G$, that extends $f_0$ and $f_1$ resticed to $\langle B C \rangle$. Since $g(e) = f_1(e)$ for $e \in \langle BC \rangle$, there is a homomorphism $h$ of $\langle ABCD\rangle $ into $G$, that extends $g$, $f_1$, and therefore $f_0$. We get $A \ind_B CD$, as desired.

\item[Mon] It follows by Lemma \ref{mon}.

\end{description}
\end{proof}

Using the Theorem \label{mueller} of I. M\"uller we obtain:

\begin{cor}\label{aut1}
Let $M_0$ be a \Frai limit. Assume that the   free amalgam of finitely generated substructures of  $M_0$ exists. Then $Aut(M_0)$ is universal for all substructures $N \subseteq M_0$.
\end{cor}

\begin{definition}
A relation $A \ind_B C$ for small subsets of a monster model $\C$ is a stationary independence relation in $\C$, if it fulfils {\bf Inv, Mon, Trans, Sym, Ex, Stat} and 
\begin{description}
\item[Fin] {\em Finite Character} $A \ind_B C$ \iff $\bar{a} \ind_B \bar{c}$ for all finite tuple $\bar{a}$ in $A$ and $\bar{c}$ in $C$.
\end{description}
\end{definition}

A stationary independence relation in $\C$ has all properties of non-forking in a stable theory except {\em Local Character}. Furthermore {\em Boundedness} is replaced by the stronger property {\bf Stat}. In the next chapter there are examples with the tree property  of the second kind.

\begin{cor}\label{aut2}
Let $L$ be finite and $\K$ be countable class of finitely generated L-structures that are uniformly locally finite. Assume a $\K$-saturated model $M_0$ exists and $\C$ is a monster model of $Th(M_0)$. 
If $\K$ is closed under free amalgamation, then the free amalgam for small substructures of $\C$ exists and defines a stationary independence relation  in $\C$. 
\end{cor}

\begin{proof} Let $\J$    be the class of all substructures of models of $Th(M_0)$. Then $\K$ is the class of the finitely generated substructures of models of $M_0$. These structures are finite.
By Lemma \ref{sigma} the  free amalgam of substructures of $\C$ exists. All properties  exept {\bf Fin} are shown in the same way as in Theorem \ref{free=ind}. {\bf Mon} implies the assertion from the left to the right of {\bf Fin}. By Lemma \ref{sigma} (1) the other direction follows from the consistency of $Th(M_0) \cup \Sigma_{\J}(A,B,C)$. We use compactness and the consistency of all $Th(M_0) \cup \Sigma_{\J}(\langle \bar{a}B \rangle, B , \langle \bar{c}B \rangle)$ for all finite $\bar{a}$ and $\bar{c}$, similary as in the proof of Lemma \ref{sigma} (2).
\end{proof}

We will apply the results of this chapter to the \Frai limits of graded Lie algebras over finite  fields  and of 
c-nilpotent groups of exponent p ($c < p$) wit extra predicates for a central Laszard series. 

\section{Graded Lie algebras over fields}
We consider graded Lie algebras $A$ over a fixed field $K$ in the language of  vector spaces over $K$ extended by a function symbol $[x,y]$ for the Lie multiplication and unary predicates $U_i$ with $1 \le i \le \omega$, such that 
\[ A = \bigoplus_{1 \le i \le \omega} A_i \] 
as a vector space, where $A_i$ is the interpretation  of $U_i$ and $[a,b] \in A_{i+j}$, if $a \in A_i$ and $b \in A_j$. We say that the elements of $A_i \setminus \{ 0 \}$ have degree i. A graded Lie algebra $A$ is c-nilpotent, if $A_i = \langle 0 \rangle$ for $c < i$. In this case we use $U_i$ only for $i \le c$.

\begin{theorem}\label{Lieam1}
The class of  c-nilpotent graded Lie algebras over a field $K$ is closed under  free amalgamation. 
\end{theorem}

\begin{proof} Let $\J$ be the class of c-nilpotent graded Lie algebras over $K$. It is $\forall$-elementary and subalgebras of finitely generated algebras in $\J$ are again finitely generated. By Lemma \ref{sigma}
it is sufficient to give a construction of  a free amalgam of $A$ and $C$ over $B$, where $A$, $B$, $C$ are finitely generated c-nilpotent graded Lie algebras over $K$ ,and $B$ is  a common subalgebra.  

We choose a vector space basis 
\[ X_B = \{b_{i,j} : 1 \le i \le c , j < \beta_i \} \] 
of $B$ with $U_i(b_{i,j})$. Then we extend $X_B$ by
\[ X_A = \{a_{i,j} : 1 \le i \le c , j < \alpha_i \}\] and 
\[ X_C = \{c_{i,j} : 1 \le i \le c , j < \gamma_i \}\]  
with $U_i(a_{i,j})$, $U_i(c_{i,j})$ and $X_A \cap X_C = \emptyset$, such that $X_A X_B$ is a vector space basis for $A$ and $X_B X_C$ is a vector space basis of $C$. Let $X$ be $X_A X_B X_C$. We use the graded set $X$ as a set of free generators of   the free graded Lie algebra $F(X)$. The elements of $X$ are in $U_i$ according to the definition above. Let $J_A$ be the ideal in $F(X_B X_A)$ generated by all equations $[x,y] = z$ in $A$ where $x,y \in X_B X_A$ and $z$ is a linear combination of elements in $X_A X_B \cap U_{i+j}$ , if $U_i(x)$ and $U_j(y)$.
Then $F(X_A X_B)/J_A$ is isomorphic to $A$. Analogously we define $J_C$ in $F(X_B X_C)$, such that 
$F(X_B X_C)/J_C$ is isomorphic to $C$. Let $J$ be the ideal in $F(X)$ generated by $J_A$ and $J_C$.

\begin{description}
\item[Claim 1] $J \cap F(X_A X_B) = J_A$ and $J \cap F(X_B X_C) = J_C$
\item[Claim 2] $F(X)/J$ is the free amalgam of $A$ and $C$ over $B$. 
\end{description}

First we show that Claim 1 implies Claim 2. By Claim 1 $F(X)/J$ contains isomorphic images $A',B',C'$ of  $A, B, C$ respectively, such that $A' \cap C' = B'$ and $\langle A',C' \rangle = F(X)/J$.  Let $f_A$ and $f_C$ be any pair of homomorphisms of $A'$ and $C'$ respectively into a c-nilpotent graded Lie algebra $G$ over $K$, such that $f_A(b) = f_C(b)$ for $b \in B'$. If we map the elements of $X$ onto their $f_A$- respectively $f_C$-images in $G$, then we get an homomorphism $f$ of $F(X)$ into $G$. The kernel of $f$ contains $J$ by the definition of $J_A$ and $J_C$. Hence $f$ induces the desired homomorphism of $F(X)/J$ into $G$.

To prove Claim 1 it is sufficient
to construct a strong amalgam directely step by step. We use again
$X_A$, $X_B$, and $X_C$, where the $\alpha_i$ , $\beta_i$, and $\gamma_i$ are finite. Now the underlying vector space of the amalgam $D$ is a vectorspace where $X = X_A X_B X_C$ is part of a basis of this space. For $x \in X$ we have $U_i(x)$ if and only if  $x = a_{i,j}$ or $x = b_{i,j}$ or $x = c_{i,j}$ for some j.  The only problem is the definition of the Lie multiplication for the elements of a vector space basis of D. Since multiplication with elements from $U_c$ gives $0$, we can put all elemtens of $U_c(X)$ into $X_B$. Therefore we assume w.l.o.g. that

\[\alpha_c = \gamma_c = 0.\]

First we  solve the following essential  case:

\begin{description}
\item[Major  Case] $X_A = \{a\}$ and $X_C = \{c\}$ with $U_i(a)$ , $U_j(c)$, and $i,j < c$.
\end{description}

Let $Y$ be a vector space basis of the free graded Lie algebra over $K$ freely generated by $a$ and $c$.
We assume that $Y$ is a set of Hall basic monomials.  $X_B Y$ will be a vectorspace basis of the amalgam $D$. The degree  of the elements of $X$ is given in $A$ and $C$. The degree of the other elements of $Y$ is canonically given using the degrees of $a$ and $c$. Finally we define the Lie multiplication in such a way that $D$ becomes a Lie algebra with the given graduation. We have only to consider  $[y,b] = - [b,y]$ for $y \in Y$ and $b \in X_B$. The problem is to ensure the Jacobin identity for all triple $[[y_1,y_2],b]$ and $[[y,b],d]$ where $y_1, y_2, y \in Y$ and $b,d \in X_B$. Inside $A$, $C$, and $\langle Y \rangle$ the Lie multiplication is given. Note that $[a,b] \in B$  and $[c,b] \in B$ for all $b\in X_B$.
This is the starting point of an inductiv definition on the degree of $y \in Y$ of $[y,b]$ for all $b \in X_B$.
If $y = [y_1,y_2]$, then we define 
\[  [[y_1,y_2],b] = [[y_1,b],y_2] + [y_1,[[y_2,b]] \]
accordingly to the Jacobi-identity. By induction $[y_1,b] = b_1 \in  B$ and $[y_2,b] = b_2 \in B$ are defined and also $[b_1, y_2] \in B$ and $[y_1, b_2] \in B$.

We have to check the Jacobi identity for the case $y \in Y$ and $b,d \in X_B$. Again we use induction on the comlexity of $y$. The begin with $y = a$ or $y = c$ is true in $A$ respectively in $C$ and therefore in $D$.
Now we assume that $y_1$ with any two elements of $X_B$  and $y_2$ with any two elements of $X_B$ satisfy the Jacobi identity. We have to show:

\begin{description}
\item[(l=r)] \[   [[[ y_1,y_2],b],d] = [[[y_1,y_2],d],b] + [[y_1,y_2][b,d]] \]
\end{description}

Using the inductiv definition of $[y,e]$ for $y \in Y$ and $e \in X_B$ the right side can be written as:

\begin{description}
\item[(r)] \[ [[[y_1,d],y_2],b] + [[y_1,[y_2,d]],b] + [[y_1,[b,d]],y_2] + [y_1,[y_2,[b,d]]] \]
\end{description}

Now we apply the definitions and the induction to the left side and obtain the  following identities. We use that $[y_1,b], [y_2,b], [y_1,d], [y_2,d] \in B$

\begin{description}
\item[($l_1)$]  \[ [[[y_1,b],y_2],d] + [[y_1,[y_2,b],d] = \]
\item[($l_2$)]  \[ [[[y_1,b],d],y_2] +  [[y_1,b], [y_2,d]] + [[y_1,d],[y_2,b]] + [y_1,[[y_2,b],d]] = \]
\item[($l_3$)]  \[ [[[y_1,d],b],y_2] +  [[[y_1,[b,d]],y_2] + [[y_1,[y_2,d]],b] + [y_1,[b,[y_2,d]]] + \]
                       \[ [[[y_1,d],y_2],b] +  [y_2,[[y_1,d],b]]] + [[y_1,[y_2,d]],b] + [y_1,[y_2,[b,d]]] .\]
\end{description}

After cancelation in ($l_3$) we see that it equal to (r) as desired. The proof for the Major Case is finished. In fact we have constructed a free amalgam. \bigskip

\begin{description}
\item[Reduction to $X_C = \{ c \}$] If the strong amalgam exist for all $A$, $B$,and $C = \langle X_B c\rangle$, then it exists for all  $A$, $B$, and $C$.
\end{description}

We assume that the assertion is true for  $X_C = \{c\}$. We show by induction on $c-i$ that the strong amalgam exists for all $X_C$ with $\gamma_j =0$ for $j < i$. The case $c = i$ is clear, since we can assume that $\alpha_c = \gamma_c = 0$, as discussed above. 

We fix $i$ and assume that the assertion is true for $i+1$. By a second induction on the size of $\gamma_i$ we reduce the problem to the case $\gamma_i = 1$. For induction step of this induction let $c = c_{\gamma_i - 1}$ and $C^-$ be the subalgebra of $C$ generated by $X_B X_C \setminus \{c \}$. By the second induction there is a strong amalgam $D^-$ of $A$ and $C^-$ over $B$. Now we have to amalgamate 
$D^-$ and $C$ over $C^-$. This is the case $X_C = \{ c \}$ and we can apply the assumption in the claim.

\bigskip

\begin{description}
\item[Reduction to the Major Case] The Major Case implies the Reduction to $X_C = \{c\}$.
\end{description}

We can apply the same arguments to all situations  $A$, $B$ and $X_C = \{ c \}$ and come to the Major Case.
The assertion of the theorem follows. Note that all amalgams condtructed are free.
\end{proof}

By compactness and Lemma \ref{sigma}(1) we obtain:

\begin{cor} The class of all graded Lie algebras over a given field $K$ is closed under free amalgamation.
\end{cor}

Note that the class of finitely generated c-nilpotent graded Lie algebras over a finite field is  countable.

\begin{cor}\label{liefrai}
Let $K$ be a finite field and $\K$ be the class of finitely generated generated c-nilpotent graded Lie algebras over $K$. Then the following is true.
\begin{enumerate}

\item A countable $\K$-saturated structure $M_0$ exists. 
\item In  $M_0$ the free amalgam of finitely generated substructures exists and  is a stationary independence relation.
\item $Aut(M_0)$ is universal for all $Aut(M)$, where $M$ is an at most countable c-nilpotent graded Lie algebra over $K$.
\end{enumerate}
\end{cor}

\begin{proof} We use Fact \ref{fact2.2}, Theorem \ref{Lieam1}, Theorem \ref{free=ind}, and Corollary \ref{aut1}. All at most countable c-nilpotent graded Lie algbras over $K$ can be embedded in $M_0$.
\end{proof}

Corollary \ref{aut2} implies:

\begin{cor}
If $K$ is a finite field, $M_0$ is the \Frai limit of all finitely generated c-nilpotent graded Lie algebras and 
$\C$ is a monster model of $Th(M_0)$, then the free amalgam defines a stationary independence relation in $\C$. The theory has the tree property of the second kind.
\end{cor}

For the tree property of the second kind see \cite{Bau13}.

\section{c-nilpotent graded assoziative  algebras} 

\begin{lemma}
c-nilpotent graded assoziative algebras do not have the amalgamation property.  
\end{lemma}
\begin{proof}
We use the language from graded Lie algebras, but for the multiplication of $x$ and $y$ we write $xy$. Let $c = 3$. We consider the free 3-nilpotent graded assoziative algebras $F$ freely generated by $a_0, a_1, b_0, b_1, b_2, b_3,b_4,b_5,c$,  $F_A $ freely generated by 
$a_0, a_1, b_0, b_1, b_2, b_3,b_4,b_5$, $F_B$ freely genrated by  $b_0, b_1, b_2, b_3,b_4,b_5$, and $F_C$ freely generated by $b_0, b_1, b_2, b_3,b_4,b_5,c$. To obtain $A$ we factorize $F_A$ by the ideal $J_A$ generated by $a_0b_0 + a_1b_1$. $C$ is obtained from $F_C$  using the ideal generated by $b_0c + b_2b_4$ and $b_1c + b_3b_5$. The images of the $b_i$'s generate in $A$ and $C$ a subalgebra isomorphic to $F_B$. we call it $B$. An amalgam would be isomorphic to $F/J$, where $J$ is the ideal generated by $J_A$ and $J_C$. But this contains a new relation in $A$:

\[  (a_0b_0 + a_1b_1)c - a_1(b_1c+ b_3b_5) - a_0(b_0c + b_2b_4) = - a_1b_3b_5 - a_0b_2b_4 \]
a contradiction.   
\end{proof}

\section{c-nilpotent groups of exponent $p > c$}

As in \cite{Bau04} we consider the class c-nilpotent groups $\G^U_{c,p}$ of exponent $p > c$ with extra unary predicates $U_1, \ldots, U_c$ for a central Laszard series. That means we have 
\[ G = U_1(G) \supseteq \ldots \supseteq U_c(G) \] 
and 
\[ \langle \bigcup_{l+k = n} [U_l(G), U_k(G)] \rangle \subseteq U_n \subseteq Z_{c+1-n}.\]

Without the extra predicates this class of groups is denoted by $\G_{c,p}$. It is well known that the amalgamation property fails for $\G_{c,p}$. In \cite{Bau04} strong amalgamation for $\G^U_{c,p}$ is shown. 
A careful analysis of the proof shows that it is a free amalgam.
The free amalgamation of c-nilpotent graded Lie algebras over $\F_p$ is used (see Theorem \ref{Lieam1} for a correct proof). Therefore we get:

\begin{theorem} 
It holds:
\begin{enumerate}
\item $\G^U_{c,p}$ has (HEP), (JEP), and the free amalgamation property.
\item The \Frai limit $G^U_0$ of the finite groups in $\G^U_{c,p}$ exits.
\item $Th(G^U_0)$ is $\aleph_0$ - categorical and allows the elimination of quantifiers.
\end{enumerate}
\end{theorem}

Let $G_0$ be the reduct of $G^U_0$ to the language of group theory. 
Since $\Gamma_n(G^U_0) = Z_{c+1-n}(G^U_0) = U_n(G^U_0)$, $U_n(G^U_0)$ is in $G_0$ 0-definable.
Furthermore every $G \in \G_{c,p}$ becomes a structure in $\G^U_{c,p}$, if we define $U_n(G) = \Gamma_n(G)$. Hence we obtain:

\begin{cor}\label{G0}
$Th(G_0)$ is $\aleph_0$ - categorical and universal for all at most countable $G \in \G_{c,p}$.
\end{cor}

By Theorem \ref{free=ind} and Corollary \ref{aut1} we obtain

\begin{theorem}
\item[(a)] In the monster model of  $Th(G^U_0)$ the  free amalgam exists and is a stationary independence relation.
\item[(b)] $Aut(G^U_0)$ is universal for all $Aut(N)$ where $(N)$ is a substructure of $G^U_0$.
\end{theorem}

The same arguments as for Corollary \ref{G0} imply

\begin{cor}
$Aut(G_0)$ is universal for all $Aut(G)$ for all at most countable groups $G \in \G_{c,p}$.
\end{cor}

\end{document}